\documentclass[11pt,oneside]{amsart}

\title{On the classification of automorphisms of trees}
\author{Kyle Beserra}
\address{Kyle Beserra, University of North Texas, 1155 Union Circle \#311430, Denton, Texas 76203}
\email{KyleBeserra@my.unt.edu}
\author{Samuel Coskey}
\address{Samuel Coskey, Boise State University, 1910 University Dr, Boise, ID 83725}
\email{scoskey@gmail.com}
\urladdr{scoskey.org}

\usepackage[marginratio=1:1]{geometry}
\usepackage{setspace}
\usepackage{mathpazo,bm,amssymb}
\usepackage{tikz}
\usetikzlibrary{decorations.markings}
\tikzset{->-/.style={decoration={markings,mark=at position 0.5 with {\arrow{>}}},postaction={decorate}}}

\newtheorem{theorem}{Theorem}[section]
\newtheorem{lemma}[theorem]{Lemma}
\newtheorem{proposition}[theorem]{Proposition}

\theoremstyle{definition}

\theoremstyle{remark}
\newtheorem*{remark}{Remark}

\DeclareMathOperator{\Aut}{Aut}
\DeclareMathOperator{\OT}{OT}
\newcommand{\F}{{\mathbb F}}
\newcommand{\Z}{{\mathbb Z}}

\usepackage{etoolbox}
\makeatletter\pretocmd{\@seccntformat}{\S}{}{}
  \pretocmd{\@subseccntformat}{\S}{}{}\makeatother

\begin{document}
\begin{abstract}
  We identify the complexity of the classification problem for automorphisms of a given countable regularly branching tree up to conjugacy. We consider both the rooted and unrooted cases. Additionally, we calculate the complexity of the conjugacy problem in the case of automorphisms of several non-regularly branching trees.
\end{abstract}
\maketitle

%%%%%%%%%%%%%%%%%%%%%%%%%%%%%%
\section{Introduction}
%%%%%%%%%%%%%%%%%%%%%%%%%%%%%%

It is a result of J.\ Tits \cite[Proposition~3.2]{tits} that every tree automorphism falls into exactly one of the following three types:
\begin{enumerate}
\item invert an edge;
\item translate a bi-infinite path; or
\item fix a subtree.
\end{enumerate}
The result is also exposited in \cite[Section~6]{serre}. While this result is both beautiful and useful, it is an incomplete classification in the sense that within each of the three classes there are still many different automorphisms that are distinct up to conjugacy equivalence. In \cite{gawron}, the authors further classify each of the three types of automorphisms of a given tree up to conjugacy using a variety of invariants such as marked trees.

In this article we investigate the \emph{complexity} of the classification of automorphisms of a given tree up to conjugacy. While the results in \cite{gawron} can easily be used to obatin upper bounds on complexity, this still leaves several key questions. First, what is the precise complexity of these invariants? And second, are these upper bounds tight? In short, how hard is it to completely classify the automorphisms of a given tree up to conjugacy?

Our main result is to identify the precise complexity of the conjugacy problem for automorphisms of countable regularly branching trees. We address the cases of both ordinary (graph-theoretic) trees and of rooted (set-theoretic) trees. We also address the complexity of each of the three types (a)--(c) separately.

In order to state these questions and results formally, it is necessary to adopt the language of invariant descriptive set theory. In this framework the objects to be classified are coded as elements of a standard Borel space, and the classification itself is identified with an equivalence relation $E$ on that space. One may then compare the complexity of classification problems using the key notion of Borel reducibility. Here, if $E,F$ are equivalence relations on standard Borel spaces $X,Y$, we say that $E$ is \emph{Borel reducible} to $F$ if there exists a Borel mapping $f\colon X\rightarrow Y$ such that
\[x\mathrel{E}x'\iff f(x)\mathrel{F}f(x')\text{.}
\]
In this case we write $E\leq_BF$.

Some of the simplest classification problems are those which are Borel reducible to the equality equivalence relation $\Delta(2^\omega)$ on the standard Borel space $2^\omega$. Such classification problems are called \emph{smooth}. Just above $\Delta(2^\omega)$ is the \emph{almost equality} relation $E_0$ on $2^\omega$ defined by $x\mathrel{E}_0x'$ iff $x(n)=x'(n)$ for all but finitely many $n$. By the Glimm--Effros dichotomy, if $E$ is a Borel equivalence relation then either $E$ is smooth or else $E_0\leq_B E$. Moving towards the higher end of the spectrum we have the isomorphism relation on the class of all countable structures in a countable language. If $E$ is bireducible with this equivalence relation then we say that $E$ is \emph{Borel complete}.

In our case the objects to be classified are the automorphisms of a fixed countable tree $T$. This is naturally a standard Borel space, as it is easy to see that $\Aut(T)$ is a Polish group with the topology of pointwise convergence. The classification we are interested in is the conjugacy equivalence relation on $\Aut(T)$. The conjugacy equivalence relation on automorphisms $\phi$ of $T$ may also be viewed as the isomorphism equivalence relation on expanded structures $(T,\phi)$. Thus the complexity of the conjugacy problem for automorphisms of $T$ is at most Borel complete.

Conjugacy problems have been investigated in the context of Borel reducibility under numerous circumstances, including operator theory (for example \cite{lupini-gardella}), ergodic theory (for example \cite{foreman-weiss}), and many others. One of the present authors has studied in \cite{ces,ce1} the conjugacy problem for automorphisms of ultrahomogeneous structures such as the random graph and the rational order. Our present study of regular trees serves to broaden the scope of that investigation, since regular trees are $1$-homogeneous but not ultrahomogeneous.

In the next section we collect several folklore results on countable rooted trees. We show that the classification of automorphisms of a finitely branching rooted tree is smooth. On the other hand the classification of automorphisms of the fully branching rooted tree is Borel complete. In the third section we study the countable regularly branching unrooted trees. We show that the conjugacy classification of automorphisms of a finitely branching regular unrooted tree is Borel bireducible with $E_\infty$, the universal countable Borel equivalence relation. We also show that the classification of automorphisms of the fully branching unrooted tree is Borel complete. In the fourth and final section we analyze the conjugacy problem for automorphisms of several non-reguarly branching trees. We consider a class of examples of non-regular rooted trees, as well as a single example of a non-regular unrooted tree. We also raise some broad questions regarding the conjugacy problem for automorphisms of a general tree or graph.

\textbf{Acknowledgement.} The work represents a portion of the first author's master's thesis \cite{kyle-thesis}. The thesis was written at Boise State University under the second author's supervision. We would like to thank John Clemens and Simon Thomas for helpful conversations and suggestions.

%%%%%%%%%%%%%%%%%%%%%%%%%%%%%%
\section{Rooted trees}
%%%%%%%%%%%%%%%%%%%%%%%%%%%%%%

In this section we calculate the complexity of the classification of countable regularly branching rooted trees. For us, a countable rooted tree is a subset $T$ of the set $\omega^{<\omega}$ of all finite sequences of natural numbers, with the property that $T$ is closed under initial segments. Two vertices of $T$ are joined by an edge if one is the immediate successor or predecessor of the other.

We show first that the automorphisms of a given finitely branching rooted tree are smoothly classifiable up to conjugacy. The classification we provide does not extend to arbitrary infinitely branching rooted trees. We show instead that the conjugacy problem for the infinitely branching regular rooted tree is Borel complete.

\begin{theorem}
  \label{thm:fb-rooted}
  Let $T$ be a finitely branching rooted tree. Then the conjugacy problem for $\Aut(T)$ is smooth.
\end{theorem}

\begin{proof}
	  First observe that the group $\Aut(T)$ is compact. Indeed, letting $T_n$ denote the truncation of $T$ to its elements of height $<n$, we can identify $\Aut(T)$ with a closed subspace of the compact space $\prod_{n\in\omega}(T_n)^{T_n}$ via the mapping $\alpha\mapsto(\alpha\restriction T_n)$. Next observe that the conjugacy relation on $\Aut(T)$ is the orbit equivalence relation induced by the conjugation action of $\Aut(T)$ on itself. It is well-known that an orbit equivalence relation induced by the continuous action of a compact group is smooth, since the space of orbits (with the Effros Borel structure) may be used as a set of complete invariants. It follows from this that the conjugacy relation on $\Aut(T)$ is smooth.
\end{proof}

\begin{remark}
  It is also possible to give a combinatorial classification for the conjugacy problem on $\Aut(T)$, as is done in \cite[Section~3]{gawron}, which is a straightforward generalization of the methods of \cite[Section~10]{dougherty-jackson-kechris}. Since the details will be used in the next section we include a brief summary here.
  
  If $T$ is a finitely branching rooted tree and $\phi\in\Aut(T)$, we define its \emph{labeled orbit tree} $\OT_\phi$ as follows. The vertices of $\OT_\phi$ are the orbits of the action of $\phi$ on $T$. For any $O,O'\in\OT_\phi$, we say that $O<O'$ if for every $x\in O$ there exists $y\in O'$ such that $x<y$. Finally we label each vertex $O\in\OT_\phi$ with the natural number $|O|$.

	  Now the isomorphism relation on the class of finitely branching labeled orbit trees is smooth. In fact by a standard application of K\"onig's tree lemma, a finitely branching labeled tree is classified by the sequence of isomorphism types of its finite truncations. Moreover we claim that the map $\phi\mapsto\OT_\phi$ is a Borel reduction from the conjugacy relation on $\Aut(T)$ to the isomorphism relation on orbit trees.

	  Indeed, if $\alpha$ conjugates $\phi$ to $\psi$, then it induces a bijection from $\phi$-orbits to $\psi$-orbits which preserves the ordering defined above, as well as the cardinality of the orbits. Conversely, if $\alpha\colon\OT_\phi\to\OT_\psi$ is an isomorphism, then by selecting for each $t\in O\in\OT_\phi$ an appropriate destination $\alpha_0(t)\in\alpha(O)\in\OT_\psi$ one can inductively construct an automorphism which conjugates $\phi$ to $\psi$. This concludes the remark.
\end{remark}

While it is also possible to classify automorphisms of a non-finitely branching tree by their orbit trees, the orbit trees themselves need not be classifiable. In fact, the next result shows that the classification of automorphisms of the regular infinitely branching rooted tree $\omega^{<\omega}$ is not smooth.

\begin{theorem}
  \label{thm:full-tree-complete}
  The conjugacy problem for $\Aut(\omega^{<\omega})$ is Borel complete.
\end{theorem}

\begin{proof}
  By \cite{friedman-stanley}, the isomorphism relation on countable rooted trees is Borel complete. Thus it suffices to show that there is a Borel reduction from the isomorphism relation on countable rooted trees to the conjugacy relation on $\Aut(\omega^{<\omega})$.

  Let $T$ be a given countable rooted tree. For the proof we choose an automorphism $\sigma$ of $\omega^{<\omega}$ with no fixed points other than the root. To begin the construction, we embed $T$ into a copy $C_T$ of $\omega^{<\omega}$ as follows. For each vertex $x\in T$ we attach a new copy $T_x$ of $\omega^{<\omega}$ to $T$ with $x$ as its root. The new successors of each $x\in T$ are interleaved with the original successors of $x$ in a straightforward way. We then let $\phi_T$ be the automorphism of $C_T$ which fixes each $x\in T$ and acts by $\sigma$ on each copy $T_x$ of $\omega^{<\omega}$.
  
  We claim that the mapping $T\mapsto\phi_T$ is a Borel reduction from the isomorphism relation on countable rooted trees to conjugacy relation on $\Aut(\omega^{<\omega})$. Indeed, if $\alpha\colon T\to T'$ is an isomorphism, then $\alpha$ naturally extends to an isomorphism $\bar\alpha\colon C_T\to C_{T'}$. Then $\phi_{T'}\circ\bar\alpha(x)=\bar\alpha(x)=\bar\alpha\circ\phi_T(x)$ for every $x\in T$, and it is easy to see that $\phi_{T'}\circ\bar\alpha(t)=\bar\alpha\circ\phi_T(t)$ for every $t\in T_x$ too. Thus $\phi_T$ is conjugate to $\phi_{T'}$.
  
  Conversely, if $\alpha$ conjugates $\phi_T$ to $\phi_{T'}$, then $\alpha$ must carry the fixed points of $\phi_T$ to the fixed points of $\phi_{T'}$. By construction, the set of fixed points of $\phi_T$ is isomorphic to $T$, and similarly for $\phi_{T'}$. It follows that $\alpha$ witnesses that $T\cong T'$, as desired.
\end{proof}

This completes the study of the classification problem for automorphisms of countable regular rooted trees. In Section~4 we will briefly address some examples of non-regular rooted trees.

%%%%%%%%%%%%%%%%%%%%%%%%%%%%%%
\section{Unrooted trees}
%%%%%%%%%%%%%%%%%%%%%%%%%%%%%%

Recall from the introduction that every automorphism $\phi$ of an unrooted tree $T$ must have exactly one of three types. We now elaborate on the three types:
\begin{enumerate}
\item $\phi$ inverts an edge, that is, there are $x,y\in T$ with $x\sim y$ and $\phi(x)=y=\phi^{-1}(x)$;
\item $\phi$ translates a bi-infinite path, that is, there are distinct $x_n\in T$ for $n\in\mathbb Z$ and $k>0$ such that $x_n\sim x_{n+1}$ and $\phi(x_n)=x_{n+k}$ for all $i$; or
\item $\phi$ fixes a subtree, that is, there exists a nonempty connected $S\subset T$ such that $\phi(x)=x$ for all $x\in S$.
\end{enumerate}
In this section we will find the complexity of the conjugacy problem for automorphisms of a regular unrooted (graph-theoretic) tree. We additionally break the space of automorphisms into the types (a)--(c), and find the complexity of conjugacy on each subset.

We begin with an observation regarding the conjugacy problem for automorphisms of type~(a) for a general finitely branching unrooted tree.

\begin{proposition}
  \label{prop:typea}
  Let $T$ be a finitely branching tree. The conjugacy problem for automorphisms of $T$ of type~(a) is smooth.
\end{proposition}

\begin{proof}
  We will use the observation in \cite[Section~5]{gawron} (together with an additional point of care) that any automorphism of $T$ that inverts an edge can be regarded as an automorphism of a related rooted tree.
  
  In detail, for $e,e'\in T$ let $e\equiv e'$ if they are conjugate, that is, if there is $\alpha\in\Aut(T)$ such that $e'=\alpha(e)$. For the rest of the proof we fix a set $S$ of equivalence class representatives for the $\equiv$ relation. For each edge $e\in S$ we let $T_e$ be the tree obtained by inserting a vertex $x_e$ into the middle of the edge $e$. We regard $T_e$ as a rooted tree with root $x_e$.

  Now if $\phi$ inverts the edge $e$, then $\phi$ induces an automorphism $\hat\phi$ of $T_{e_0}$ where $e_0$ is the $\equiv$-representative of $e$. Then it is not difficult to check that $\phi\mapsto\hat\phi$ is a reduction from conjugacy of automorphisms of $T$ of type~(a) to the disjoint union over $e\in S$ of the conjugacy relations on $\Aut(T_e)$. Since the trees $T_e$ are rooted, it follows from Theorem~\ref{thm:fb-rooted} that the conjugacy relation on $\Aut(T)$ is smooth.
  
  Lastly we comment that we can arrange for the map $\phi\mapsto\hat\phi$ to be Borel. To accomplish, this we need only fix in advance the countable representative set $S$ together with a countable family of automorphisms $\alpha_e$ such that for each edge $e$ of $T$, $\alpha_e$ maps $e$ to its representative in $S$.
\end{proof}

For the remainder of the section, we confine ourselves to the case when $T$ is a regular unrooted tree. For each $n$ such that $2\leq n\leq\omega$ we let $R_n$ denote a fixed copy of the unique unrooted tree such that every vertex has degree exactly $n$.

We can already fully describe the complexity of the conjugacy problem for automorphisms of type~(a) for each of the trees $R_n$. The conjugacy relation for automorphisms of $R_2$ of type~(a) clearly has a single equivalence class, in other words, it is Borel bireducible with $\Delta(1)$. Next, the conjugacy problem for automorphisms of $R_n$ of type~(a) for $3\leq n<\omega$ is Borel bireducible with $\Delta(2^\omega)$. Indeed, the $\leq_B$ direction was just shown in Proposition~\ref{prop:typea}. For the reverse reduction, using the terminology of the remark following Theorem~\ref{thm:fb-rooted}, there are continuum many orbit trees and one can simply find an automorphism to represent each. Finally, the conjugacy problem for automorphisms of $R_\omega$ of type~(a) is Borel complete. Here one can again use a construction from orbit trees, together with the previously stated fact that the classification of all countable rooted trees is Borel complete.

The next result records the complexity of the conjugacy problem for automorphisms of $R_n$ of type~(b).

\begin{proposition}
  \label{prop:typeb}
  Let $2\leq n\leq\omega$. The conjugacy problem for automorphisms of $R_n$ of type~(b) is smooth. In fact it is Borel bireducible with $\Delta(\omega)$.
\end{proposition}

\begin{proof}
  This is a direct corollary of \cite[Theorem~5.3]{gawron}, which states that an automorphism of type~(b) is determined up to conjugacy by the ``amplitude'' of the translation, that is, the minimum value of $k$ in the definition of type~(b), together with the conjugacy class of the bi-infinite path of the translation. Since there is just one conjugacy class of bi-infinite paths of $R_n$, the result follows.
  %
  % We will show automorphisms that translate a bi-infinite path are completely determined up to conjugacy by the ``amplitude'' of the translation, that is, the value of $k$ in the definition of type~(b).
  %
  % It is clear that if two automorphisms of type~(b) are conjugate, then they have the same amplitude. For the converse, let $\phi,\psi$ be automorphisms of $R_n$ of type~(b) with the same amplitude $k$. We wish to show there is an automorphism $\alpha$ which conjugates $\phi$ to $\psi$. To begin, by applying a conjugation, we may assume that $\phi,\psi$ translate the same oriented bi-infinite path $Z$ in the positive direction by $k$ units. Indeed, for any other oriented bi-infinite path $Z'$, there exists an automorphism of $R_n$ mapping $Z'$ to $Z$.
  %
  % Next, we may further assume that $k=1$. Indeed, if $k>1$, we would regard each of $\phi,\psi$ as a disjoint union of $k$ many translations $\phi_i$ and $\psi_i$ of amplitude $1$. Given conjugators from $\phi_i$ to $\psi_i$, we would amalgamate them to find a single conjugator from $\phi$ to $\psi$.
  %
  % To deal with the case $k=1$, we enumerate $Z=\{z_i\mid i\in\Z\}$ and let $T_i$ denote the subtree rooted at $z_i$ and otherwise disjoint from $Z$. Let $\alpha_0$ be the identity on $T_0$ and define $\alpha_i$ on $T_i$ by $\alpha_i=\psi^i\alpha_0\phi^{-i}$. Then by design we have that $\alpha_{i+1}\phi=\psi\alpha_i$, which implies that $\alpha$ conjugates $\phi$ to $\psi$. This completes the proof.
\end{proof}

We now turn to the automorphisms of $R_n$ of type~(c).

\begin{theorem}
  \label{thm:e-infty}
  If $3\leq n<\omega$ then the conjugacy problem for automorphisms of $R_n$ of type (c) is Borel bireducible with $E_\infty$.
\end{theorem}

Here we recall that $E_\infty$ is the orbit equivalence relation induced by the free group on two generators $\F_2$ acting by left-translation on its power set $\mathcal P(\F_2)$ (with the product topology of $2^{\F_2}$). The relation $E_\infty$ is universal among countable Borel equivalence relations, that is, Borel equivalence relations where each equivalence class is countable. It is well-known (see \cite{jackson-kechris-louveau}) that $E_\infty$ is Borel bireducible with the isomorphism relation on the class of finitely branching trees. Before proving Theorem~\ref{thm:e-infty}, we need the following refinement of this latter result.

\begin{lemma}
  \label{lem:onethree}
  $E_\infty$ is Borel bireducible with the isomorphism relation on the class of countable trees with the property that every vertex has degree $1$ or $3$.
\end{lemma}

\begin{proof}
  It suffices to give a Borel reduction from the orbit equivalence relation of $\F_2$ acting on $\mathcal P(\F_2)$ to the isomorphism equivalence relation on countable trees with degree $1$ or $3$. For this we use a technical modification of the argument from \cite{jackson-kechris-louveau} mentioned above.

	Given a subset $S$ of $\F_2$, we construct a tree $T_S$ as follows. To begin, denote the generators of $\F_2$ by $a,b$, and let $\Gamma$ be the corresponding directed, labeled Cayley graph.	We define $T_S$ by replacing the vertices and directed edges of $\Gamma$ with the widgets shown in Figure~\ref{fig:construction}.
  
  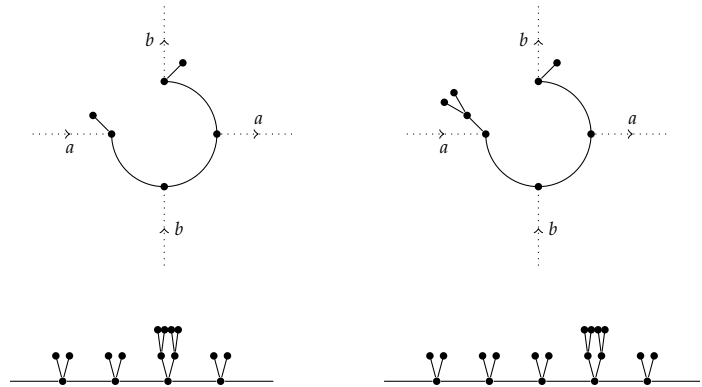
\begin{figure}[ht]
    \begin{center}
      \begin{tikzpicture}[vertex/.style={circle,fill,inner sep=1pt},scale=.7,baseline=0]
        \node[vertex] (a) at (-1,0) {};
        \node[vertex] (b) at (0,-1) {};
        \node[vertex] (c) at (1,0) {};
        \node[vertex] (d) at (0,1) {};
        \draw (-1,0) arc[radius=1,start angle=-180,end angle=90];
        \draw (a) -- +(135:.5) node[vertex] {};
        \draw (d) -- +(45:.5) node[vertex] {};
        \draw[->-,dotted] (-2.5,0)--node[below]{\tiny$a$} (a);
        \draw[->-,dotted] (0,-2.5)--node[right]{\tiny$b$} (b);
        \draw[->-,dotted] (c)--node[above]{\tiny$a$} (2.5,0);
        \draw[->-,dotted] (d)--node[left]{\tiny$b$} (0,2.5);
      \end{tikzpicture}
      \hspace{.5in}
      \begin{tikzpicture}[vertex/.style={circle,fill,inner sep=1pt},scale=.7,baseline=0]
        \node[vertex] (a) at (-1,0) {};
        \node[vertex] (b) at (0,-1) {};
        \node[vertex] (c) at (1,0) {};
        \node[vertex] (d) at (0,1) {};
        \draw (-1,0) arc[radius=1,start angle=-180,end angle=90];
        \draw (a) -- +(135:.5) node[vertex] (a1) {};
        \draw (d) -- +(45:.5) node[vertex] {};
        \draw (a1) -- +(120:.5) node[vertex] {};
        \draw (a1) -- +(150:.5) node[vertex] {};
        \draw[->-,dotted] (-2.5,0)--node[below]{\tiny$a$} (a);
        \draw[->-,dotted] (0,-2.5)--node[right]{\tiny$b$} (b);
        \draw[->-,dotted] (c)--node[above]{\tiny$a$} (2.5,0);
        \draw[->-,dotted] (d)--node[left]{\tiny$b$} (0,2.5);
      \end{tikzpicture}
      \vspace{.3in}

      \begin{tikzpicture}[vertex/.style={circle,fill,inner sep=1pt},scale=.7,baseline=0]
        \node[vertex] (a) at (0,0) {};
        \node[vertex] (b) at (1,0) {};
        \node[vertex] (c) at (2,0) {};
        \node[vertex] (d) at (3,0) {};
        \draw (a) -- +(75:.5) node[vertex] {};
        \draw (a) -- +(105:.5) node[vertex] {};
        \draw (b) -- +(75:.5) node[vertex] {};
        \draw (b) -- +(105:.5) node[vertex] {};
        \draw (c) -- +(75:.5) node[vertex] (c1) {};
        \draw (c) -- +(105:.5) node[vertex] (c2) {};
        \draw (d) -- +(75:.5) node[vertex] {};
        \draw (d) -- +(105:.5) node[vertex] {};
        \draw (c1) -- +(82.5:.5) node[vertex] {};
        \draw (c1) -- +(97.5:.5) node[vertex] {};
        \draw (c2) -- +(82.5:.5) node[vertex] {};
        \draw (c2) -- +(97.5:.5) node[vertex] {};
        \draw (-1,0)--(a)--(b)--(c)--(d)--(4,0);
        % \draw[->-,dotted] (1,-1)--node[below]{\tiny$a$}(2,-1);
      \end{tikzpicture}
      \hspace{.5in}
      \begin{tikzpicture}[vertex/.style={circle,fill,inner sep=1pt},scale=.7,baseline=0]
        \node[vertex] (o) at (-1,0) {};
        \node[vertex] (a) at (0,0) {};
        \node[vertex] (b) at (1,0) {};
        \node[vertex] (c) at (2,0) {};
        \node[vertex] (d) at (3,0) {};
        \draw (o) -- +(75:.5) node[vertex] {};
        \draw (o) -- +(105:.5) node[vertex] {};
        \draw (a) -- +(75:.5) node[vertex] {};
        \draw (a) -- +(105:.5) node[vertex] {};
        \draw (b) -- +(75:.5) node[vertex] {};
        \draw (b) -- +(105:.5) node[vertex] {};
        \draw (c) -- +(75:.5) node[vertex] (c1) {};
        \draw (c) -- +(105:.5) node[vertex] (c2) {};
        \draw (d) -- +(75:.5) node[vertex] {};
        \draw (d) -- +(105:.5) node[vertex] {};
        \draw (c1) -- +(82.5:.5) node[vertex] {};
        \draw (c1) -- +(97.5:.5) node[vertex] {};
        \draw (c2) -- +(82.5:.5) node[vertex] {};
        \draw (c2) -- +(97.5:.5) node[vertex] {};
        \draw (-2,0)--(o)--(a)--(b)--(c)--(d)--(4,0);
        % \draw[->-,dotted] (.5,-1)--node[below]{\tiny$b$}(1.5,-1);
      \end{tikzpicture}      
      \caption{Structures in $T_S$ that code the vertices and directed, labeled edges of $\Gamma$. Top left: a vertex of $\Gamma$ not in $S$. Top right: a vertex of $\Gamma$ in $S$. Bottom left: a left-to-right directed edge in $\Gamma$ labeled $a$. Bottom right: a left-to-right directed edge in $\Gamma$ labeled $b$.}
      \label{fig:construction}
    \end{center}
  \end{figure}

  It is evident from the diagrams that every vertex of $T_S$ has degree $1$ or $3$. Moreover it is not difficult to verify that the vertices and directed labeled edges of $\Gamma$ may be recovered from the structure of $T_S$. For example, the only place one can find two adjacent vertices, each not adjacent to a leaf, is in a structure coding a vertex of $\Gamma$. Moreover there are two such pairs if the coded vertex lies in $S$. From here it is possible to identify the position of each element of the coding structure.
	  
	We claim that the mapping $S\to T_S$ gives the desired reduction. Clearly if there exists $g\in\F_2$ such that $gS=S'$, then $g$ induces an isomorphism $T_S\cong T_{S'}$. Conversely if $\alpha\colon T_S\to T_{S'}$ is an isomorphism, then the above remarks imply that $\alpha$ preserves the coding structures and therefore induces an automorphism of $\Gamma$ which carries $S$ to $S'$. Since the only automorphisms of a directed, labeled Cayley graph are the translations, we conclude that there exists $g\in\F_2$ such that $gS=S'$, as desired.
\end{proof}

\begin{remark}
  Using extra leaves in the coding structures, one can similarly show that for any $n<\omega$, $E_\infty$ is Borel bireducible with the isomorphism relation on countable trees such that every vertex has degree $1$ or $n$. This is carried out in detail in \cite{kyle-thesis}.
\end{remark}

We are now ready to complete the proof that the conjugacy relation on type~(c) elements of $\Aut(R_n)$ is Borel bireducible with $E_\infty$.

\begin{proof}[Proof of Theorem~\ref{thm:e-infty}]
  To see that the conjugacy relation on the type~(c) elements of $\Aut(R_n)$ is Borel reducible to $E_\infty$, first note that for $n<\omega$ the group $\Aut(R_n)$ is locally compact. Indeed a neighborhood basis at the identity consists of the pointwise stabilizers of finite sets, and it is easy to see from the arguments of Theorem~\ref{thm:fb-rooted} that such stabilizers are compact. We may now appeal to the well-known result from \cite{kechris-sections} which implies that any orbit equivalence relation induced by a continuous action of a locally compact group is Borel reducible to $E_\infty$. We remark that it is also possible to produce a combinatorial reduction using the methods of \cite[Section~5]{gawron}.

  For the reverse direction, by Lemma~\ref{lem:onethree} and the remark following it, it suffices to find a Borel reduction from the isomorphism relation on countable trees such that every vertex has degree $1$ or $n$ to the conjugacy relation on $\Aut(R_n)$. For this, we will use a construction similar to Theorem~\ref{thm:full-tree-complete}.
  
  Let $T$ be a given tree such that every vertex has degree $1$ or $n$. We embed $T$ into a copy $C_T$ of $R_n$ by adding a copy of $(n-1)^{<\omega}$ to each leaf of $T$. We then construct an automorphism $\phi_T$ of $C_T$ as follows. Fix in advance an automorphism $\sigma$ of $(n-1)^{<\omega}$ which moves every point except the root. Then let $\phi_T$ be the automorphism which fixes each $x\in T$ and acts by $\sigma$ on each $T_x$. One can then argue just as in the proof of Theorem~\ref{thm:full-tree-complete} that $T\cong T'$ if and only if $\phi_T$ is conjugate to $\phi_{T'}$.
\end{proof}

We remark that the conjugacy relation on automorphisms of $R_2$ of type~(c) has just two equivalence classes; formally it is Borel bireducible with $\Delta(2)$. We conclude this section by recording the complexity in the infinitely branching case.

\begin{theorem}
  The conjugacy problem for automorphisms of $R_\omega$ of type~(c) is Borel complete.
\end{theorem}

\begin{proof}
  The isomorphism relation on countable unrooted trees is Borel complete, so it suffices to find a Borel reduction from the isomorphism relation on such trees to the conjugacy relation on the type~(c) elements of $\Aut(R_\omega)$. For this we again adapt the proof of Theorem~\ref{thm:full-tree-complete}. Given a countable unrooted tree $T$, we embed it in $R_\omega$ in such a way that every vertex $x\in T$ has infinitely many neighbors not in $T$. We then build in a uniform way an automorphism $\phi_T$ of $R_\omega$ whose fixed point set is exactly $T$. We then verify as before that $T\cong T'$ iff $\phi_T$ is conjugate to $\phi_{T'}$.
\end{proof}

%%%%%%%%%%%%%%%%%%%%%%%%%%%%%%
\section{Further examples}
%%%%%%%%%%%%%%%%%%%%%%%%%%%%%%

In this section we broaden our study to include trees which are not regularly branching. In doing so, we will see some phenomena that did not occur in the case of regular trees. For instance, all of the conjugacy problems we have seen so far have the complexity of $\Delta(2^\omega)$ or lower, $E_\infty$, or are Borel complete. Our examples below will exhibit several additional levels in the Borel complexity hierarchy.

Another theme in previous sections is an apparent coincidence between the complexity of the conjugacy problem for $\Aut(T)$ and the complexity of the isomorphism classification of subtrees of $T$. For instance, the isomorphism classification of finitely branching rooted trees is smooth, and the isomorphism classification of arbitrarily branching trees is Borel complete (see \cite{friedman-stanley}). These facts parallel the results of Section~2. Similarly the complexity of the isomorphism classification of finitely branching unrooted trees is $E_\infty$ (see \cite{jackson-kechris-louveau}), which parallels the results of Section~3. Again, our examples below will exhibit a departure from this pattern.

In the next result, we consider the conjugacy problem for automorphisms of the rooted trees $\omega^{\leq n}$, that is, the full rooted trees of given bounded height $n$.

In order to state the result, we recall that for an equivalence relation $E$ on a standard Borel space $X$, the \emph{jump} of $E$, denoted $E^+$, is defined on the space $X^\omega$ by:
\[x\mathrel{E}^+x'\iff \{[x(n)]_E:n\in\omega\}=\{[x'(n)]_E:n\in\omega\}
\]
% That is, $x,x'$ enumerate the same set of $E$-classes
We further let $E^{+0}=E$ and let $E^{+n}=(E^{+n-1})^+$ for $n>0$. By \cite{friedman-stanley}, the sequence of equivalence relations $\Delta(2^\omega)^{+n}$ increases strictly in Borel complexity.

\begin{proposition}
  For all $n\geq1$, the conjugacy problem for automorphisms of $\omega^{\leq n}$ has the complexity $\Delta(2^\omega)^{+n-1}$.
\end{proposition}

\begin{proof}[Proof outline]
  To begin, an automorphism of $\omega^{\leq1}$ is simply a permutation of $\omega$. It is well-known that the permutations of $\omega$ are classified up to conjugacy by their cycle type, and it easily follows that this conjugacy problem is Borel bireducible with $\Delta(2^\omega)$.
  
  Next, an automorphism of $\omega^{\leq2}$ is analogous to an automorphism of infinitely many disjoint copies of the complete graph on infinitely many vertices. By \cite[Theorem~3.2]{ce1}, the conjugacy problem for automorphisms of this latter graph is Borel bireducible with $\Delta(2^\omega)^+$.
  
  Proceeding inductively, an automorphism of $\omega^{\leq n+1}$ may be viewed as an automorphism of infinitely many disjoint copies of $\omega^{\leq n}$. Inspecting the details of the argument of \cite[Theorem~3.2]{ce1}, it is not difficult to show that the conjugacy relation on $\Aut(\omega^{\leq n+1})$ is Borel bireducible with the jump of the conjugacy relation on $\Aut(\omega^{\leq n})$. This completes the proof.
\end{proof}

By contrast, it is well-known (and not difficult to check) that the isomorphism problem for subtrees of $\omega^{\leq n}$ has the complexity $\Delta(2^\omega)^{+n-2}$.

We close this section with an example of a non-regular unrooted tree such that the complexity of the conjugacy problem is distinct from all previous examples. Recall that $E_0$ denotes the equivalence relation on $2^\omega$ defined by $x\mathrel{E}_0y$ iff $x(n)=y(n)$ for all but finitely many $n$. It is a standard fact that $E_0$ is Borel bireducible with the equivalence relation $E_\Z$ induced by the translation action of $\Z$ on $\mathcal P(\Z)$. The complexity of the relations $E_0$ and $E_{\Z}$ is just above that of the smooth equivalence relations (in the formal sense known as the generalized Glimm--Effros dichotomy), and well below $E_\infty$. We refer the reader to \cite[Chapter~6]{gao} for more about $E_0$.

\begin{proposition}
  \label{prop:tz}
  Let $T_\Z$ denote the tree shown in the Figure~\ref{fig:Ztree}. The conjugacy problem for automorphisms of $T_\Z$ of types~(a) and~(b) are smooth, and the conjugacy problem for automorphism of $T_\Z$ of type~(c) is Borel bireducible with $E_0$.
\end{proposition}

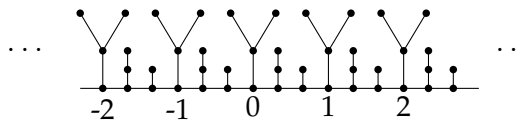
\begin{figure}[ht]
  \begin{tikzpicture}[every node/.style={circle,fill,inner sep=1pt}]
    \foreach \i in {-2,...,2} {
      \draw (\i,0)
      node[label={[label distance=-2pt,below]\i}]{}
      -- (\i,.5) node{} -- (\i-.3,1) node{};
      \draw (\i,.5) -- (\i+.3,1) node{};
      \draw (\i+.33,0) node{} -- (\i+.33,.25) node{} -- (\i+.33,.5) node{};
      \draw (\i+.66,0) node{} -- (\i+.66,.25) node{};
    }
    \draw (-2.3,0)--(3,0);
    \node[fill=none] at (-3,.5) {$\cdots$};
    \node[fill=none] at (3.5,.5) {$\cdots$};
  \end{tikzpicture}
  \caption{The tree $T_{\mathbb Z}$.\label{fig:Ztree}}
\end{figure}

\begin{proof}
  The claim about automorphisms of $T_\Z$ of type~(a) follows directly from Proposition~\ref{prop:typea}. The claim about automorphisms of $T_\Z$ of type~(b) follows from the argument of Proposition~\ref{prop:typeb}, as $T_\Z$ possesses a unique bi-infinite path.

  For the claim about automorphisms of type~(c), we first show that $E_\Z$ is Borel reducible to the conjugacy relation on automorphisms of $T_\Z$ of type~(c). Given a subset $A\subset\mathbb Z$ we let $\phi_A$ be the automorphism of $T_\Z$ which exchanges the two degree~$1$ nodes in the $Y$-shape corresponding to each $n\in A$, and fixes all nodes in the $Y$-shape corresponding to each $n\notin A$. Then $\phi_A$ is of type~(c), and it is easy to check that this mapping satisfies $A\mathrel{E}_\Z A'$ iff $\phi_A\sim\phi_{A'}$.

  To establish the Borel reduction in the converse direction, it is sufficient to show that the mapping $A\mapsto\phi_A$ is bijective. Indeed, it is easy to see that it is one-to-one. To see that it is onto, first observe that any automorphism of $T_\Z$ induces a translation on the $Y$-shapes of $T_\Z$. Moreover, any type~(c) automorphism must induce the trivial translation. Hence every type~(c) automorphism consists of swaps of the degree~$1$ nodes of $Y$-shapes of $T_\Z$, and is of the form $\phi_A$ for some $A\subset\Z$. This completes the proof.
\end{proof}

We close by conjecturing that one can find countable trees whose conjugacy problems are of a variety of distinct complexities.

\bibliographystyle{alpha}
\bibliography{bib}

\end{document}